\documentclass[12pt,halfline,a4paper]{ouparticle}
\usepackage{graphicx}
\usepackage{multirow}
\usepackage{adjustbox}
\usepackage{booktabs}
\usepackage{lscape}
\usepackage{float}
\usepackage{amsmath,amssymb,amsfonts}
\usepackage{amsthm}
\usepackage{array}
\usepackage{cite}
\usepackage{mathrsfs}
\usepackage[title]{appendix}
\usepackage{xcolor}
\usepackage{textcomp}
\usepackage{manyfoot}
\usepackage{booktabs}
\usepackage{algorithm}
\usepackage{algorithmicx}
\usepackage{algpseudocode}
\usepackage{listings}
\newtheorem{theorem}{Theorem}

\begin{document}

\title{Vertex energy distributions in regular graph structures}

\author{%
\name{H.M. Nagesh}
\address{Department of Science and Humanities \\
PES University, Bangalore, Karnataka, India}
\email{hmnagesh1982@gmail.com}
\and
\name{U. Vijaya Chandra Kumar}
\address{Department of Mathematics, School of Applied Sciences \\
REVA University, Bangalore, Karnataka, India}
\email{upparivijay@gmail.com}
\and
\name{N. Narahari}
\address{Department Mathematics \\
University College of Science \\ Tumkur University, Tumakuru, Karnataka, India}
\email{narahari\_nittur@yahoo.com}
}
\abstract{The energy of a vertex $v_i$ in a graph $G$ is defined as $\mathcal{E}_G(v_i) = |A|_{ii}$, where $A$ is the adjacency matrix of $G$, $A^*$ denotes the conjugate transpose of $A$, and $|A| = (AA^*)^{1/2}$. The total energy of the graph, $\mathcal{E}(G)$, is then the sum of the energies of all vertices: $\mathcal{E}(G) = \mathcal{E}_G(v_1) + \mathcal{E}_G(v_2) + \dots + \mathcal{E}_G(v_n)$. In this paper, we compute the vertex energy for several well-known regular graphs, including the Frucht graph, Desargues graph, Tutte-Coxeter graph, Heawood graph, Shrikhande graph, and Petersen graph.
}

\date{}


\maketitle

\section{Introduction}
\label{sec1}
A simple graph \( G \) with \( n \) vertices and \( m \) edges has a vertex set given by \( V(G) = \{ v_1, v_2, \dots, v_n \} \). The degree of a vertex \( v \in V(G) \) is the number of edges that are incident to $v$. If every vertex in \( G \) has the same degree \( r \), then \( G \) is called an \( r \)-regular graph. The adjacency matrix of a graph \( G \) is a square matrix \( A = A(G) = [a_{ij}] \) of order \( n \), where \( a_{ij} = 1 \) if the vertices \( v_i \) and \( v_j \) are adjacent, and \( a_{ij} = 0 \) otherwise. 

Let $\phi(G: \lambda) = \det(\lambda I - A(G))$ is the characteristic polynomial of \( G \), where \( I \) is the identity matrix of order \( n \); and \( \lambda_1, \lambda_2, \dots, \lambda_n \) are the eigenvalues of \( G \), which are the eigenvalues of the adjacency matrix \( A(G) \). Let \( \mathcal{E}(G) \) \ represents the energy of a graph \( G \), which is the total of the absolute values of its eigenvalues \([1]\). That is, 
\[
\mathcal{E}(G) = \sum_{i=1}^{n} |\lambda_i|.
\]
More details on graph energy can be found in \([2]\).  

Recently, Arizmendi and Juarez-Romero \([3]\) first proposed the idea of a vertex's energy. According to them, the energy of a vertex \( v_i \), denoted by \( E_G(v_i) \), is defined as  
\[
\mathcal{E}_G(v_i) = |A|_{ii} \quad \text{for } 1 \leq i \leq n,  \text{and } |A| = (AA^T)^{1/2}.
\]

Since \( \mathcal{E}(G) = \text{trace}(|A(G)|) \), the energy of a graph can be computed by summing the energies of its vertices:
\[
\mathcal{E}(G) = \sum_{i=1}^{n} \mathcal{E}_G(v_i).
\]

Arizmendi et al. \cite{4} explored the fundamental properties of a vertex's energy, including certain bounds, and demonstrated in \cite{5} that it can be computed using a Coulson integral formula. Furthermore, \cite{4} provides the energy values for vertices in various graphs, such as transitive graphs, complete graphs, hypercubes, cycles, complete bipartite graphs, friendship graphs, dandelion graphs, and paths. The vertex energies of subdivision graphs, including the complete graph, complete bipartite graph, cocktail party graph, and Petersen graph, are examined in \cite{6}. Very recently, Gutman et al. \cite{7} presented a procedure for determining the energy of vertices in graphs, stating that this computation requires not only the eigenvalues but also the eigenvectors of the graph's adjacency matrix. In this paper, we determine the vertex energies of several regular graphs, including the Frucht graph, Desargues graph, Tutte–Coxeter graph, Heawood graph, Shrikhande graph, and Petersen graph. 

\section{Main results}
In this section, we determine the vertex energies of several regular graphs, including the Frucht graph, Desargues graph, Tutte–Coxeter graph, Heawood graph, Shrikhande graph, and Petersen graph. 

The following results are required for this analysis.

\begin{theorem}[\cite{4}]
Let \( G \) be a graph with \( n \) vertices. Then the vertex energy of vertex \( v_i \) is given by
\begin{equation}
\mathcal{E}_G(v_i) = \sum_{j=1}^{n} p_{ij} |\lambda_j|, \quad i = 1, \dots, n,
\end{equation}
where \( \lambda_j \) denotes the \( j \)-th eigenvalue of the adjacency matrix \( A \), and the weights \( p_{ij} \) satisfy
\[
\sum_{i=1}^{n} p_{ij} = 1, \quad \sum_{j=1}^{n} p_{ij} = 1.
\]
Moreover, \( p_{ij} = u_{ij}^2 \), where \( U = (u_{ij}) \) is the orthogonal matrix whose columns are the eigenvectors of \( A \).
\end{theorem}

\begin{theorem}[\cite{4}]
Let \( G \) be a graph with \( n \) vertices. For \( k \in \mathbb{N} \), let \( \varphi_i(A^k) \) be the \( k \)-th moment of \( A \) with respect to the linear functional \( \varphi_i \). Then
\begin{equation}
\varphi_i(A^k) = \sum_{j=1}^{n} p_{ij} \lambda_j^k, \quad i = 1, \dots, n,
\end{equation}
where, the notation is the same as in Theorem 1.
\end{theorem}

Recall that \( \varphi_i(A^k) \) is equal to the number of \( v_i \)-\( v_i \) walks in \( G \) of length \( k \).

\subsection{Vertex energies of Frucht graph}
\begin{theorem}
Let $G$ be the Frucht graph with vertex set \( V(G) = \{v_1, v_2, \ldots, v_{12}\} \). 
Then, the graph energy at each vertex \( v_i \in V(G) \) is given by:\\
$\mathcal{E}_G(v_1) = 1.50636,  
\mathcal{E}_G(v_2) = 1.55632, 
\mathcal{E}_G(v_3) = 1.45627, 
\mathcal{E}_G(v_4) = 1.44865, 
\mathcal{E}_G(v_5) = 1.54705, 
\mathcal{E}_G(v_6) = 1.52488, 
\mathcal{E}_G(v_7) = 1.48642, 
\mathcal{E}_G(v_8) = 1.54800, 
\mathcal{E}_G(v_9) = 1.43233, 
\mathcal{E}_G(v_{10}) = 1.44129, 
\mathcal{E}_G(v_{11}) = 1.55632, 
\mathcal{E}_G(v_{12}) = 1.56952.$
\end{theorem}
\begin{proof}
The Frucht graph $G$ is uniquely constructed to possess only the identity automorphism, meaning it has a trivial automorphism group. Consequently, its adjacency matrix typically exhibits simple eigenvalues (each with multiplicity one), resulting in a greater number of distinct eigenvalues. The eigenvalues of the Frucht graph \cite{8} are: \\
$\lambda = \{-2.33866, -2, -1.80194, -1.45106, -1, -0.44504, 0, 0.51912, 1.24698, 2, 2.2706, 3\}$. 

We now compute the energy at each vertex $v_{i}$, where $1 \leq i \leq 12$. Since there are 12 distinct eigenvalues, the weights $p_{ij}$ in Eq. (1)  can be calculated with a 12 by 12 system of linear equations coming from calculating the first moments, on one hand, directly counting walks and on the other. For convenience, we denote the weights \( p_{ij} \) as \( p_{i,j} \).

\textbf{Graph energy calculation for vertex \(v_1\):} \\
Let \( p_{1,1}, p_{1,2}, p_{1,3}, p_{1,4}, p_{1,5}, p_{1,6}, p_{1,7}, p_{1,8}, p_{1,9}, p_{1,10}, p_{1,11}, p_{1,12}\) be the weights of the vertex \( v_1 \) in \( G \). Then, by Eq. (2), 
\[
\tiny
\begin{aligned}
p_{1,1} \lambda_1^0 + p_{1,2} \lambda_2^0 + p_{1,3} \lambda_3^0 + p_{1,4} \lambda_4^0 + p_{1,5} \lambda_5^0 + p_{1,6} \lambda_6^0 + p_{1,7} \lambda_7^0 + p_{1,8} \lambda_8^0 + p_{1,9} \lambda_9^0 + p_{1,10} \lambda_{10}^0 + p_{1,11} \lambda_{11}^0 + p_{1,12} \lambda_{12}^0 =1,\\
p_{1,1} \lambda_1^1 + p_{1,2} \lambda_2^1 + p_{1,3} \lambda_3^1 + p_{1,4} \lambda_4^1 + p_{1,5} \lambda_5^1 + p_{1,6} \lambda_6^1 + p_{1,7} \lambda_7^1 + p_{1,8} \lambda_8^1 + p_{1,9} \lambda_9^1 + p_{1,10} \lambda_{10}^1 + p_{1,11} \lambda_{11}^1 + p_{1,12} \lambda_{12}^1 =0,\\ 
p_{1,1} \lambda_1^2 + p_{1,2} \lambda_2^2 + p_{1,3} \lambda_3^2 + p_{1,4} \lambda_4^2 + p_{1,5} \lambda_5^2 + p_{1,6} \lambda_6^2 + p_{1,7} \lambda_7^2 + p_{1,8} \lambda_8^2 + p_{1,9} \lambda_9^2 + p_{1,10} \lambda_{10}^2 + p_{1,11} \lambda_{11}^2 + p_{1,12} \lambda_{12}^2 =3,\\
p_{1,1} \lambda_1^3 + p_{1,2} \lambda_2^3 + p_{1,3} \lambda_3^3 + p_{1,4} \lambda_4^3 + p_{1,5} \lambda_5^3 + p_{1,6} \lambda_6^3 + p_{1,7} \lambda_7^3 + p_{1,8} \lambda_8^3 + p_{1,9} \lambda_9^3 + p_{1,10} \lambda_{10}^3 + p_{1,11} \lambda_{11}^3 + p_{1,12} \lambda_{12}^3=2, \\
p_{1,1} \lambda_1^4 + p_{1,2} \lambda_2^4 + p_{1,3} \lambda_3^4 + p_{1,4} \lambda_4^4 + p_{1,5} \lambda_5^4 + p_{1,6} \lambda_6^4 + p_{1,7} \lambda_7^4 + p_{1,8} \lambda_8^4 + p_{1,9} \lambda_9^4 + p_{1,10} \lambda_{10}^4 + p_{1,11} \lambda_{11}^4 + p_{1,12} \lambda_{12}^4 =15,\\
p_{1,1} \lambda_1^5 + p_{1,2} \lambda_2^5 + p_{1,3} \lambda_3^5 + p_{1,4} \lambda_4^5 + p_{1,5} \lambda_5^5 + p_{1,6} \lambda_6^5 + p_{1,7} \lambda_7^5 + p_{1,8} \lambda_8^5 + p_{1,9} \lambda_9^5 + p_{1,10} \lambda_{10}^5 + p_{1,11} \lambda_{11}^5 + p_{1,12} \lambda_{12}^5 = 22,\\
p_{1,1} \lambda_1^6 + p_{1,2} \lambda_2^6 + p_{1,3} \lambda_3^6 + p_{1,4} \lambda_4^6 + p_{1,5} \lambda_5^6 + p_{1,6} \lambda_6^6 + p_{1,7} \lambda_7^6 + p_{1,8} \lambda_8^6 + p_{1,9} \lambda_9^6 + p_{1,10} \lambda_{10}^6 + p_{1,11} \lambda_{11}^6 + p_{1,12} \lambda_{12}^6 =95,\\
p_{1,1} \lambda_1^7 + p_{1,2} \lambda_2^7 + p_{1,3} \lambda_3^7 + p_{1,4} \lambda_4^7 + p_{1,5} \lambda_5^7 + p_{1,6} \lambda_6^7 + p_{1,7} \lambda_7^7 + p_{1,8} \lambda_8^7 + p_{1,9} \lambda_9^7 + p_{1,10} \lambda_{10}^7 + p_{1,11} \lambda_{11}^7 + p_{1,12} \lambda_{12}^7 =200,\\
p_{1,1} \lambda_1^8 + p_{1,2} \lambda_2^8 + p_{1,3} \lambda_3^8 + p_{1,4} \lambda_4^8 + p_{1,5} \lambda_5^8 + p_{1,6} \lambda_6^8 + p_{1,7} \lambda_7^8 + p_{1,8} \lambda_8^8 + p_{1,9} \lambda_9^8 + p_{1,10} \lambda_{10}^8 + p_{1,11} \lambda_{11}^8 + p_{1,12} \lambda_{12}^8 =701,\\
p_{1,1} \lambda_1^9 + p_{1,2} \lambda_2^9 + p_{1,3} \lambda_3^9 + p_{1,4} \lambda_4^9 + p_{1,5} \lambda_5^9 + p_{1,6} \lambda_6^9 + p_{1,7} \lambda_7^9 + p_{1,8} \lambda_8^9 + p_{1,9} \lambda_9^9 + p_{1,10} \lambda_{10}^9 + p_{1,11} \lambda_{11}^9 + p_{1,12} \lambda_{12}^9 =1756, \\
p_{1,1} \lambda_1^{10} + p_{1,2} \lambda_2^{10} + p_{1,3} \lambda_3^{10} + p_{1,4} \lambda_4^{10} + p_{1,5} \lambda_5^{10} + p_{1,6} \lambda_6^{10} + p_{1,7} \lambda_7^{10} + p_{1,8} \lambda_8^{10} + p_{1,9} \lambda_9^{10} + p_{1,10} \lambda_{10}^{10} + p_{1,11} \lambda_{11}^{10} + p_{1,12} \lambda_{12}^{10} =5653,\\
p_{1,1} \lambda_1^{11} + p_{1,2} \lambda_2^{11} + p_{1,3} \lambda_3^{11} + p_{1,4} \lambda_4^{11} + p_{1,5} \lambda_5^{11} + p_{1,6} \lambda_6^{11} + p_{1,7} \lambda_7^{11} + p_{1,8} \lambda_8^{11} + p_{1,9} \lambda_9^{11} + p_{1,10} \lambda_{10}^{11} + p_{1,11} \lambda_{11}^{11} + p_{1,12} \lambda_{12}^{11}=15422.
\end{aligned}
\]
By substituting the eigenvalues and solving this system of equations, we obtain the weights assigned for the vertex \(v_1\) as follows:\\
\[
\begin{aligned}
p_{1,1} &= 0.0278,   & p_{1,2} &= 0,   & p_{1,3} &= 0.23192, \\
p_{1,4} &= 0.16277,  & p_{1,5} &= 0.02777,  & p_{1,6} &= 0.01415, \\
p_{1,7} &= 0.16666,  & p_{1,8} &= 0.01313,  & p_{1,9} &= 0.11107, \\
p_{1,10} &= 0.03175, & p_{1,11} &= 0.12964, & p_{1,12} &= 0.08333.
\end{aligned}
\]
Therefore, by Eq. (1)
\[
\begin{aligned}
\mathcal{E}_{G}(v_1) &= 0.0278|\lambda_1| + 0|\lambda_2| + 0.23192|\lambda_3| + 0.16277|\lambda_4| + 0.02777|\lambda_5| \\
&\quad + 0.01415|\lambda_6| + 0.16666|\lambda_7| + 0.01313|\lambda_8| + 0.11107|\lambda_9| \\
&\quad + 0.03175|\lambda_{10}| + 0.12964|\lambda_{11}| + 0.08333|\lambda_{12}| \\
& = 0.0278|-2.33866| + 0|-2| + 0.23192|-1.80194| + 0.16277|-1.45106| \\  
&\quad + 0.02777|-1| + 0.01415|-0.44504| +0.16666|0|+ 0.01313|0.51912| + 0.11107|1.24698| \\
&\quad + 0.03175|2| + 0.12964|2.2706| + 0.08333|3| \\
&= 0.0278(2.33866) + 0(2) + 0.23192(1.80194) + 0.16277(1.45106) + 0.02777(1) \\
&\quad + 0.01415(0.44504) +0.16666 (0)+0.01313(0.51912) + 0.11107(1.24698) \\
&\quad + 0.03175(2) + 0.12964(2.2706) + 0.08333(3) \\
& = 1.50636.
\end{aligned}
\]

\newpage
\textbf{Graph energy calculation for vertex \(v_2\):} \\
Let \( p_{2,1}, p_{2,2}, p_{2,3}, p_{2,4}, p_{2,5}, p_{2,6}, p_{2,7}, p_{2,8}, p_{2,9}, p_{2,10}, p_{2,11}, p_{2,12} \) 
be the weights of the vertex \( v_2 \) in \( G \). Then, by Eq. (2),
\[
\tiny
\begin{aligned}
p_{2,1} \lambda_1^0 + p_{2,2} \lambda_2^0 + p_{2,3} \lambda_3^0 + p_{2,4} \lambda_4^0 + p_{2,5} \lambda_5^0 + p_{2,6} \lambda_6^0 + p_{2,7} \lambda_7^0 + p_{2,8} \lambda_8^0 + p_{2,9} \lambda_9^0 + p_{2,10} \lambda_{10}^0 + p_{2,11} \lambda_{11}^0 + p_{2,12} \lambda_{12}^0 &= 1, \\
p_{2,1} \lambda_1^1 + p_{2,2} \lambda_2^1 + p_{2,3} \lambda_3^1 + p_{2,4} \lambda_4^1 + p_{2,5} \lambda_5^1 + p_{2,6} \lambda_6^1 + p_{2,7} \lambda_7^1 + p_{2,8} \lambda_8^1 + p_{2,9} \lambda_9^1 + p_{2,10} \lambda_{10}^1 + p_{2,11} \lambda_{11}^1 + p_{2,12} \lambda_{12}^1 &= 0, \\
p_{2,1} \lambda_1^2 + p_{2,2} \lambda_2^2 + p_{2,3} \lambda_3^2 + p_{2,4} \lambda_4^2 + p_{2,5} \lambda_5^2 + p_{2,6} \lambda_6^2 + p_{2,7} \lambda_7^2 + p_{2,8} \lambda_8^2 + p_{2,9} \lambda_9^2 + p_{2,10} \lambda_{10}^2 + p_{2,11} \lambda_{11}^2 + p_{2,12} \lambda_{12}^2 &= 3, \\
p_{2,1} \lambda_1^3 + p_{2,2} \lambda_2^3 + p_{2,3} \lambda_3^3 + p_{2,4} \lambda_4^3 + p_{2,5} \lambda_5^3 + p_{2,6} \lambda_6^3 + p_{2,7} \lambda_7^3 + p_{2,8} \lambda_8^3 + p_{2,9} \lambda_9^3 + p_{2,10} \lambda_{10}^3 + p_{2,11} \lambda_{11}^3 + p_{2,12} \lambda_{12}^3 &= 2, \\
p_{2,1} \lambda_1^4 + p_{2,2} \lambda_2^4 + p_{2,3} \lambda_3^4 + p_{2,4} \lambda_4^4 + p_{2,5} \lambda_5^4 + p_{2,6} \lambda_6^4 + p_{2,7} \lambda_7^4 + p_{2,8} \lambda_8^4 + p_{2,9} \lambda_9^4 + p_{2,10} \lambda_{10}^4 + p_{2,11} \lambda_{11}^4 + p_{2,12} \lambda_{12}^4 &= 15, \\
p_{2,1} \lambda_1^5 + p_{2,2} \lambda_2^5 + p_{2,3} \lambda_3^5 + p_{2,4} \lambda_4^5 + p_{2,5} \lambda_5^5 + p_{2,6} \lambda_6^5 + p_{2,7} \lambda_7^5 + p_{2,8} \lambda_8^5 + p_{2,9} \lambda_9^5 + p_{2,10} \lambda_{10}^5 + p_{2,11} \lambda_{11}^5 + p_{2,12} \lambda_{12}^5 &= 20, \\
p_{2,1} \lambda_1^6 + p_{2,2} \lambda_2^6 + p_{2,3} \lambda_3^6 + p_{2,4} \lambda_4^6 + p_{2,5} \lambda_5^6 + p_{2,6} \lambda_6^6 + p_{2,7} \lambda_7^6 + p_{2,8} \lambda_8^6 + p_{2,9} \lambda_9^6 + p_{2,10} \lambda_{10}^6 + p_{2,11} \lambda_{11}^6 + p_{2,12} \lambda_{12}^6 &= 95, \\
p_{2,1} \lambda_1^7 + p_{2,2} \lambda_2^7 + p_{2,3} \lambda_3^7 + p_{2,4} \lambda_4^7 + p_{2,5} \lambda_5^7 + p_{2,6} \lambda_6^7 + p_{2,7} \lambda_7^7 + p_{2,8} \lambda_8^7 + p_{2,9} \lambda_9^7 + p_{2,10} \lambda_{10}^7 + p_{2,11} \lambda_{11}^7 + p_{2,12} \lambda_{12}^7 &= 182, \\
p_{2,1} \lambda_1^8 + p_{2,2} \lambda_2^8 + p_{2,3} \lambda_3^8 + p_{2,4} \lambda_4^8 + p_{2,5} \lambda_5^8 + p_{2,6} \lambda_6^8 + p_{2,7} \lambda_7^8 + p_{2,8} \lambda_8^8 + p_{2,9} \lambda_9^8 + p_{2,10} \lambda_{10}^8 + p_{2,11} \lambda_{11}^8 + p_{2,12} \lambda_{12}^8 &= 697, \\
p_{2,1} \lambda_1^9 + p_{2,2} \lambda_2^9 + p_{2,3} \lambda_3^9 + p_{2,4} \lambda_4^9 + p_{2,5} \lambda_5^9 + p_{2,6} \lambda_6^9 + p_{2,7} \lambda_7^9 + p_{2,8} \lambda_8^9 + p_{2,9} \lambda_9^9 + p_{2,10} \lambda_{10}^9 + p_{2,11} \lambda_{11}^9 + p_{2,12} \lambda_{12}^9 &= 1638, \\
p_{2,1} \lambda_1^{10} + p_{2,2} \lambda_2^{10} + p_{2,3} \lambda_3^{10} + p_{2,4} \lambda_4^{10} + p_{2,5} \lambda_5^{10} + p_{2,6} \lambda_6^{10} + p_{2,7} \lambda_7^{10} + p_{2,8} \lambda_8^{10} + p_{2,9} \lambda_9^{10} + p_{2,10} \lambda_{10}^{10} + p_{2,11} \lambda_{11}^{10} + p_{2,12} \lambda_{12}^{10} &= 5603, \\
p_{2,1} \lambda_1^{11} + p_{2,2} \lambda_2^{11} + p_{2,3} \lambda_3^{11} + p_{2,4} \lambda_4^{11} + p_{2,5} \lambda_5^{11} + p_{2,6} \lambda_6^{11} + p_{2,7} \lambda_7^{11} + p_{2,8} \lambda_8^{11} + p_{2,9} \lambda_9^{11} + p_{2,10} \lambda_{10}^{11} + p_{2,11} \lambda_{11}^{11} + p_{2,12} \lambda_{12}^{11} &= 14732.
\end{aligned}
\]

By substituting the eigenvalues and solving this system of equations, we obtain the weights assigned for the vertex \(v_2\) as follows:\\
\[
\begin{aligned}
p_{2,1} &= 0.03896, & p_{2,2} &= 0.125, & p_{2,3} &= 0.07143, \\
p_{2,4} &= 0.01829, & p_{2,5} &= 0.25, & p_{2,6} &= 0.07143, \\
p_{2,7} &= 0.04166, & p_{2,8} &= 0.02065, & p_{2,9} &= 0.07143, \\
p_{2,10} &= 0.16072, & p_{2,11} &= 0.0471, & p_{2,12} &= 0.08333.
\end{aligned}
\]
Therefore, by Eq. (1)
\[
\begin{aligned}
\mathcal{E}_{G}(v_2) &= 0.03896|\lambda_1| + 0.125|\lambda_2| + 0.07143|\lambda_3| + 0.01829|\lambda_4| \\
&\quad + 0.25|\lambda_5| + 0.07143|\lambda_6| + 0.04166|\lambda_7| + 0.02065|\lambda_8| \\
&\quad + 0.07143|\lambda_9| + 0.16072|\lambda_{10}| + 0.04710|\lambda_{11}| + 0.08333|\lambda_{12}| \\
&= 0.03896(2.33866) + 0.125(2) + 0.07143(1.80194) + 0.01829(1.45106) \\
&\quad + 0.25(1) + 0.07143(0.44504) + 0.04166(0) + 0.02065(0.51912) \\
&\quad + 0.07143(1.24698) + 0.16072(2) + 0.04710(2.2706) + 0.08333(3) \\
&= 0.09109 + 0.25000 + 0.12871 + 0.02654 \\
&\quad + 0.25000 + 0.03179 + 0 + 0.01072 \\
&\quad + 0.08909 + 0.32144 + 0.10691 + 0.24999 \\
&= 1.55632
\end{aligned}
\]

\newpage
\textbf{Graph energy calculation for vertex \( v_3 \):} \\
Let \( p_{3,1}, p_{3,2}, p_{3,3}, p_{3,4}, p_{3,5}, p_{3,6}, p_{3,7}, p_{3,8}, p_{3,9}, p_{3,10}, p_{3,11}, p_{3,12} \) be the weights of the vertex \( v_3 \) in \( G \). Then, by Eq. (2),
\[
\tiny
\begin{aligned}
p_{3,1} \lambda_1^0 + p_{3,2} \lambda_2^0 + p_{3,3} \lambda_3^0 + p_{3,4} \lambda_4^0 + p_{3,5} \lambda_5^0 + p_{3,6} \lambda_6^0 + p_{3,7} \lambda_7^0 + p_{3,8} \lambda_8^0 + p_{3,9} \lambda_9^0 + p_{3,10} \lambda_{10}^0 + p_{3,11} \lambda_{11}^0 + p_{3,12} \lambda_{12}^0 &= 1, \\
p_{3,1} \lambda_1^1 + p_{3,2} \lambda_2^1 + p_{3,3} \lambda_3^1 + p_{3,4} \lambda_4^1 + p_{3,5} \lambda_5^1 + p_{3,6} \lambda_6^1 + p_{3,7} \lambda_7^1 + p_{3,8} \lambda_8^1 + p_{3,9} \lambda_9^1 + p_{3,10} \lambda_{10}^1 + p_{3,11} \lambda_{11}^1 + p_{3,12} \lambda_{12}^1 &= 0, \\
p_{3,1} \lambda_1^2 + p_{3,2} \lambda_2^2 + p_{3,3} \lambda_3^2 + p_{3,4} \lambda_4^2 + p_{3,5} \lambda_5^2 + p_{3,6} \lambda_6^2 + p_{3,7} \lambda_7^2 + p_{3,8} \lambda_8^2 + p_{3,9} \lambda_9^2 + p_{3,10} \lambda_{10}^2 + p_{3,11} \lambda_{11}^2 + p_{3,12} \lambda_{12}^2 &= 3, \\
p_{3,1} \lambda_1^3 + p_{3,2} \lambda_2^3 + p_{3,3} \lambda_3^3 + p_{3,4} \lambda_4^3 + p_{3,5} \lambda_5^3 + p_{3,6} \lambda_6^3 + p_{3,7} \lambda_7^3 + p_{3,8} \lambda_8^3 + p_{3,9} \lambda_9^3 + p_{3,10} \lambda_{10}^3 + p_{3,11} \lambda_{11}^3 + p_{3,12} \lambda_{12}^3 &= 0, \\
p_{3,1} \lambda_1^4 + p_{3,2} \lambda_2^4 + p_{3,3} \lambda_3^4 + p_{3,4} \lambda_4^4 + p_{3,5} \lambda_5^4 + p_{3,6} \lambda_6^4 + p_{3,7} \lambda_7^4 + p_{3,8} \lambda_8^4 + p_{3,9} \lambda_9^4 + p_{3,10} \lambda_{10}^4 + p_{3,11} \lambda_{11}^4 + p_{3,12} \lambda_{12}^4 &= 17, \\
p_{3,1} \lambda_1^5 + p_{3,2} \lambda_2^5 + p_{3,3} \lambda_3^5 + p_{3,4} \lambda_4^5 + p_{3,5} \lambda_5^5 + p_{3,6} \lambda_6^5 + p_{3,7} \lambda_7^5 + p_{3,8} \lambda_8^5 + p_{3,9} \lambda_9^5 + p_{3,10} \lambda_{10}^5 + p_{3,11} \lambda_{11}^5 + p_{3,12} \lambda_{12}^5 &= 8, \\
p_{3,1} \lambda_1^6 + p_{3,2} \lambda_2^6 + p_{3,3} \lambda_3^6 + p_{3,4} \lambda_4^6 + p_{3,5} \lambda_5^6 + p_{3,6} \lambda_6^6 + p_{3,7} \lambda_7^6 + p_{3,8} \lambda_8^6 + p_{3,9} \lambda_9^6 + p_{3,10} \lambda_{10}^6 + p_{3,11} \lambda_{11}^6 + p_{3,12} \lambda_{12}^6 &= 111, \\
p_{3,1} \lambda_1^7 + p_{3,2} \lambda_2^7 + p_{3,3} \lambda_3^7 + p_{3,4} \lambda_4^7 + p_{3,5} \lambda_5^7 + p_{3,6} \lambda_6^7 + p_{3,7} \lambda_7^7 + p_{3,8} \lambda_8^7 + p_{3,9} \lambda_9^7 + p_{3,10} \lambda_{10}^7 + p_{3,11} \lambda_{11}^7 + p_{3,12} \lambda_{12}^7 &= 116, \\
p_{3,1} \lambda_1^8 + p_{3,2} \lambda_2^8 + p_{3,3} \lambda_3^8 + p_{3,4} \lambda_4^8 + p_{3,5} \lambda_5^8 + p_{3,6} \lambda_6^8 + p_{3,7} \lambda_7^8 + p_{3,8} \lambda_8^8 + p_{3,9} \lambda_9^8 + p_{3,10} \lambda_{10}^8 + p_{3,11} \lambda_{11}^8 + p_{3,12} \lambda_{12}^8 &=  799, \\
p_{3,1} \lambda_1^9 + p_{3,2} \lambda_2^9 + p_{3,3} \lambda_3^9 + p_{3,4} \lambda_4^9 + p_{3,5} \lambda_5^9 + p_{3,6} \lambda_6^9 + p_{3,7} \lambda_7^9 + p_{3,8} \lambda_8^9 + p_{3,9} \lambda_9^9 + p_{3,10} \lambda_{10}^9 + p_{3,11} \lambda_{11}^9 + p_{3,12} \lambda_{12}^9 &= 1280, \\
p_{3,1} \lambda_1^{10} + p_{3,2} \lambda_2^{10} + p_{3,3} \lambda_3^{10} + p_{3,4} \lambda_4^{10} + p_{3,5} \lambda_5^{10} + p_{3,6} \lambda_6^{10} + p_{3,7} \lambda_7^{10} + p_{3,8} \lambda_8^{10} + p_{3,9} \lambda_9^{10} + p_{3,10} \lambda_{10}^{10} + p_{3,11} \lambda_{11}^{10} + p_{3,12} \lambda_{12}^{10} &= 6209, \\
p_{3,1} \lambda_1^{11} + p_{3,2} \lambda_2^{11} + p_{3,3} \lambda_3^{11} + p_{3,4} \lambda_4^{11} + p_{3,5} \lambda_5^{11} + p_{3,6} \lambda_6^{11} + p_{3,7} \lambda_7^{11} + p_{3,8} \lambda_8^{11} + p_{3,9} \lambda_9^{11} + p_{3,10} \lambda_{10}^{11} + p_{3,11} \lambda_{11}^{11} + p_{3,12} \lambda_{12}^{11} &= 12796.
\end{aligned}
\]

Substituting the eigenvalues, we obtain the weights:
\[
\begin{aligned}
p_{3,1} &= 0.19139, & p_{3,2} &= 0.12499, & p_{3,3} &= 0.00001, \\
p_{3,4} &= 0.00190, & p_{3,5} &= 0.02778, & p_{3,6} &= 0, \\
p_{3,7} &= 0.04167, & p_{3,8} &= 0.39769, & p_{3,9} &= 0, \\
p_{3,10} &= 0.09722, & p_{3,11} &= 0.03402, & p_{3,12} &= 0.08333.
\end{aligned}
\]
Therefore, by Eq. (1)
\[
\begin{aligned}
\mathcal{E}_{G}(v_3) &= 0.19139|\lambda_1| + 0.12499|\lambda_2| + 0.00001|\lambda_3| + 0.00190|\lambda_4| \\
&\quad + 0.02778|\lambda_5| + 0|\lambda_6| + 0.04167|\lambda_7| + 0.39769|\lambda_8| \\
&\quad + 0|\lambda_9| + 0.09722|\lambda_{10}| + 0.03402|\lambda_{11}| + 0.08333|\lambda_{12}| \\
& = 0.19139 \times 2.33866 + 0.12499 \times 2.0 + 0.00001 \times 1.80194 \\
&\quad + 0.00190 \times 1.45106 + 0.02778 \times 1.0 + 0 \times 0.44504 \\
&\quad + 0.04167 \times 0.0 + 0.39769 \times 0.51912 + 0 \times 1.24698 \\
&\quad + 0.09722 \times 2.0 + 0.03402 \times 2.2706 + 0.08333 \times 3 \\
&= 0.4476 + 0.24998 + 0.000018 + 0.002757 + 0.02778 + 0.20646 \\
&\quad + 0.19444 + 0.07721 + 0.24999 \\
&= 1.45627
\end{aligned}
\]

Detailed computations for the vertex energies of the first three vertices, $v_1$, $v_2$, and $v_3$, have been presented. The remaining cases follow an analogous computational approach, involving systematic but extensive calculations. Vertex energies for the remaining vertices $v_i$ ($4 \leq i \leq 12$) can be computed using the same methodology.

\newpage

For completeness, we provide the corresponding values of \( \varphi_i(A^k) \), as defined in Eq. (2), which are essential for computing the vertex energies of the Frucht graph. These values are summarized in Table 1. 

\begin{center}
\begin{adjustbox}{width=\textwidth}
\begin{tabular}{|c|*{16}{>{\centering\arraybackslash}p{3.2em}|}}
\hline
$k$ & $v_1$ & $v_2$ & $v_3$ & $v_4$ & $v_5$ & $v_6$ & $v_7$ & $v_8$ & $v_9$ & $v_{10}$ & $v_{11}$ & $v_{12}$ \\
\hline
0  & 1.0  & 1.0  & 1.0  & 1.0  & 1.0  & 1.0  & 1.0  & 1.0  & 1.0  & 1.0  & 1.0  & 1.0  \\
1  & 0.0  & 0.0  & 0.0  & 0.0  & 0.0  & 0.0  & 0.0  & 0.0  & 0.0  & 0.0  & 0.0  & 0.0  \\
2  & 3.0  & 3.0  & 3.0  & 3.0  & 3.0  & 3.0  & 3.0  & 3.0  & 3.0  & 3.0  & 3.0  & 3.0  \\
3  & 2.0  & 2.0  & 0.0  & 2.0  & 2.0  & 2.0  & 2.0  & 2.0  & 0.0  & 2.0  & 2.0  & 0.0  \\
4  & 15.0 & 15.0 & 17.0 & 17.0 & 15.0 & 15.0 & 15.0 & 15.0 & 17.0 & 17.0 & 15.0 & 15.0 \\
5  & 22.0 & 20.0 & 8.0  & 20.0 & 22.0 & 20.0 & 22.0 & 22.0 & 6.0  & 20.0 & 20.0 & 8.0  \\
6  & 95.0 & 95.0 & 111.0& 111.0& 95.0 & 93.0 & 93.0 & 95.0 & 113.0& 113.0& 95.0 & 97.0 \\
7  & 200.0& 182.0& 116.0& 184.0& 200.0& 182.0& 200.0& 198.0& 100.0& 180.0& 182.0& 120.0\\
8  & 701.0& 697.0& 799.0& 799.0& 699.0& 677.0& 679.0& 699.0& 823.0& 821.0& 697.0& 721.0\\
9  & 1756.0&1638.0&1280.0&1658.0&1758.0&1636.0&1756.0&1736.0&1178.0&1614.0&1638.0&1324.0\\
10 & 5653.0&5603.0&6209.0&6205.0&5627.0&5461.0&5487.0&5627.0&6401.0&6373.0&5603.0&5797.0\\
11 & 15422.0&14732.0&12796.0&14870.0&15452.0&14708.0&15422.0&15286.0&12188.0&14532.0&14732.0&13132.0\\
\hline
\end{tabular}
\end{adjustbox}
\end{center}

The following are the weights corresponding to each vertex:
\[
\text{Weights of vertex } v_4:
\begin{aligned}
p_{4,1}  &= 0.0987,  & p_{4,2}  &= 0.125,   & p_{4,3}  &= 0.022,   \\
p_{4,4}  &= 0.01074, & p_{4,5}  &= 0.02777, & p_{4,6}  &= 0.36065, \\
p_{4,7}  &= 0.04166, & p_{4,8}  &= 0.0004,  & p_{4,9}  &= 0.04593, \\
p_{4,10} &= 0.00199, & p_{4,11} &= 0.18183, & p_{4,12} &= 0.08333.
\end{aligned}
\]

\[
\text{Weights of vertex } v_5:
\begin{aligned}
p_{5,1}  &= 0.00951, & p_{5,2}  &= 0.125,   & p_{5,3}  &= 0.07142, \\
p_{5,4}  &= 0.21572, & p_{5,5}  &= 0.02777, & p_{5,6}  &= 0.07143, \\
p_{5,7}  &= 0.04166, & p_{5,8}  &= 0.1108,  & p_{5,9}  &= 0.07143, \\
p_{5,10} &= 0.04961, & p_{5,11} &= 0.12231, & p_{5,12} &= 0.08333.
\end{aligned}
\]

\[
\text{Weights of vertex } v_6:
\begin{aligned}
p_{6,1}  &= 0.02174, & p_{6,2}  &= 0.125,   & p_{6,3}  &= 0.11107, \\
p_{6,4}  &= 0.08991, & p_{6,5}  &= 0.02778, & p_{6,6}  &= 0.23193, \\
p_{6,7}  &= 0.04167, & p_{6,8}  &= 0.00895, & p_{6,9}  &= 0.01415, \\
p_{6,10} &= 0.24008, & p_{6,11} &= 0.0044,  & p_{6,12} &= 0.08333.
\end{aligned}
\]

\[
\text{Weights of vertex } v_7:
\begin{aligned}
p_{7,1}  &= 0.00249,  & p_{7,2}  &= 0.0,     & p_{7,3}  &= 0.36063, \\
p_{7,4}  &= 0.0271,   & p_{7,5}  &= 0.02777, & p_{7,6}  &= 0.04594, \\
p_{7,7}  &= 0.16666,  & p_{7,8}  &= 0.05678, & p_{7,9}  &= 0.022,   \\
p_{7,10} &= 0.12699, & p_{7,11} &= 0.0803,  & p_{7,12} &= 0.08333.
\end{aligned}
\]

\[
\text{Weights of vertex } v_8:
\begin{aligned}
p_{8,1}  &= 0.02033,  & p_{8,2}  &= 0.125,    & p_{8,3}  &= -0.00001, \\
p_{8,4}  &= 0.30921,  & p_{8,5}  &= 0.02778,  & p_{8,6}  &= 0.0,      \\
p_{8,7}  &= 0.04166,  & p_{8,8}  &= 0.19488,  & p_{8,9}  &= 0.0,      \\
p_{8,10} &= 0.09722,  & p_{8,11} &= 0.10058,  & p_{8,12} &= 0.08333.
\end{aligned}
\]

\[
\text{Weights of vertex } v_9:
\begin{aligned}
p_{9,1}  &= 0.26169, & p_{9,2}  &= -0.00001, & p_{9,3}  &= 0.01416, \\
p_{9,4}  &= 0.001,   & p_{9,5}  &= 0.02778,  & p_{9,6}  &= 0.11107, \\
p_{9,7}  &= 0.16667, & p_{9,8}  &= 0.02679,  & p_{9,9}  &= 0.23193, \\
p_{9,10} &= 0.03175, & p_{9,11} &= 0.04385,  & p_{9,12} &= 0.08333.
\end{aligned}
\]

\[
\text{Weights of vertex } v_{10}:
\begin{aligned}
p_{10,1}  &= 0.15585, & p_{10,2}  &= 0.0,     & p_{10,3}  &= 0.0,      \\
p_{10,4}  &= 0.07317, & p_{10,5}  &= 0.24999, & p_{10,6}  &= 0.00001, \\
p_{10,7}  &= 0.16666, & p_{10,8}  &= 0.08259, & p_{10,9}  &= 0.0,      \\
p_{10,10} &= 0.0,     & p_{10,11} &= 0.18839, & p_{10,12} &= 0.08333.
\end{aligned}
\]

\[
\text{Weights of vertex } v_{11}:
\begin{aligned}
p_{11,1}  &= 0.03896, & p_{11,2}  &= 0.125,   & p_{11,3}  &= 0.07143, \\
p_{11,4}  &= 0.01829, & p_{11,5}  &= 0.25,    & p_{11,6}  &= 0.07143, \\
p_{11,7}  &= 0.04166, & p_{11,8}  &= 0.02065, & p_{11,9}  &= 0.07143, \\
p_{11,10} &= 0.16072, & p_{11,11} &= 0.0471,  & p_{11,12} &= 0.08333.
\end{aligned}
\]

\[
\text{Weights of vertex } v_{12}:
\begin{aligned}
p_{12,1}  &= 0.13258, & p_{12,2}  &= 0.125,   & p_{12,3}  &= 0.04594, \\
p_{12,4}  &= 0.07193, & p_{12,5}  &= 0.02778, & p_{12,6}  &= 0.022,   \\
p_{12,7}  &= 0.04167, & p_{12,8}  &= 0.06671, & p_{12,9}  &= 0.36064, \\
p_{12,10} &= 0.00198, & p_{12,11} &= 0.02046, & p_{12,12} &= 0.08333.
\end{aligned}
\]
Finally, the energy for each vertex \( v_{i}, 4 \leq i \leq 12 \), is given by 
\begin{center} 
$\mathcal{E}_G(v_4) = 1.44865$, 
$\mathcal{E}_G(v_5) = 1.54705$, 
$\mathcal{E}_G(v_6) = 1.52488$, 
$\mathcal{E}_G(v_7) = 1.48642$, 
$\mathcal{E}_G(v_8) = 1.54800$, 
$\mathcal{E}_G(v_9) = 1.43233$, 
$\mathcal{E}_G(v_{10}) = 1.44129$, 
$\mathcal{E}_G(v_{11}) = 1.55632$, 
$\mathcal{E}_G(v_{12}) = 1.56952$.
\end{center}
\end{proof}

\newpage

The distribution of vertex energies of the Frucht graph is shown in Figure 1.

 \begin{figure}[hbt!]
 	\centering
 	\includegraphics[width=160mm]{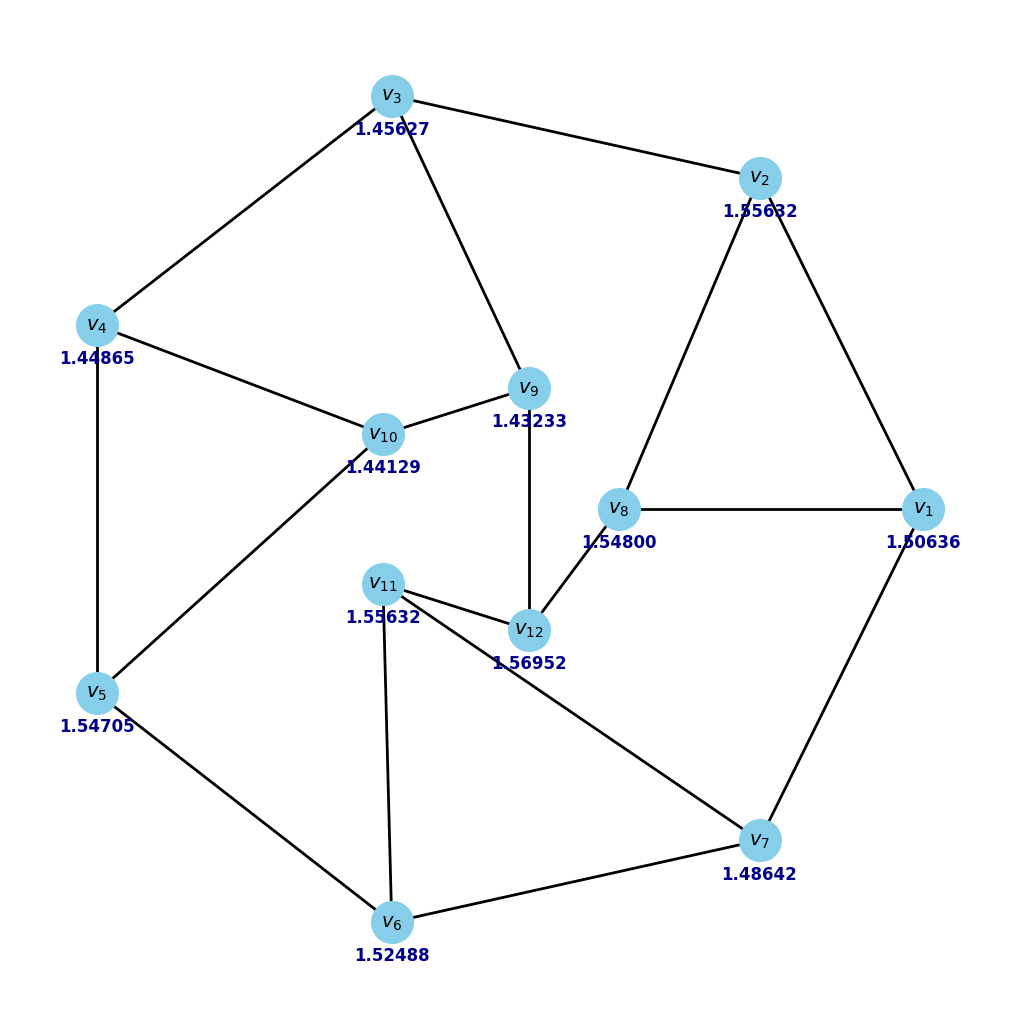}	
 	\caption{Vertex energies of Frucht graph.}		
 \end{figure}

Recall that a simple graph $G$ is said to be \emph{vertex‐transitive} if, for any two vertices $v$ and $w$, there exists an automorphism of $G$ sending $v$ to $w$.  Examples include the Desargues graph, the Tutte–Coxeter graph, the Heawood graph, the Shrikhande graph, and the Petersen graph.  

By Theorem 1, the energy of a vertex $v_i$ is
\[
  \mathcal{E}_G(v_i)
  \;=\;
  \sum_{j=1}^{n} p_{ij}\,\bigl|\lambda_j\bigr|,
\]
where the weights $p_{ij}=u_{ij}^2$ satisfy
\[
  \sum_{i=1}^n p_{ij} \;=\; 1,
  \quad
  \sum_{j=1}^n p_{ij} \;=\; 1.
\]
In a vertex‐transitive graph each $p_{ij}$ is independent of $i$, and by 
$\sum_{i=1}^n p_{ij}=1$ one obtains 
\[
  p_{ij}=\frac1n
  \quad\text{for all }i,j.
\]
Hence for every $i$,
\[
  \mathcal{E}_G(v_i)
  =
  \sum_{j=1}^n \frac1n\,\bigl|\lambda_j\bigr|
  =
  \frac{1}{n}\sum_{j=1}^n\bigl|\lambda_j\bigr|
  =
  \frac{E(G)}{n},
\]
where $E(G)=\sum_{j=1}^n|\lambda_j|$ is the total graph energy.  Thus one may compute $E(G)$ via Equations (1) and (2) and then divide by $n$ to recover the common per‐vertex energy.

We denote by \( p_{ij} \) the weights assigned to the vertices \( v_i \) for \( i \geq n \).

\subsection{Vertex energies of Desargues graph}
\begin{theorem}
Let $G$ be the Desargues graph with vertex set \( V(G) = \{v_1, v_2, \ldots, v_{20}\} \). Then $\mathcal{E}_{G}(v_{i}) = 1.6, \, 1 \leq i \leq 20$.    
\end{theorem}
\begin{proof}
The characteristic polynomial of the Desargues graph \( G \) \cite{9} is given by 
\[
\phi(G; \lambda) = (\lambda +3)(\lambda +2)^4(\lambda+1)^5(\lambda -1)^5(\lambda -2)^4(\lambda -3).
\]    

Let \( p_{11}, p_{12}, p_{13}, p_{14}, p_{15}, p_{16} \) denote the weights of the vertex \( v_1 \) in the graph \( G \). Then, according to Theorem 2, the following system of linear equations is obtained:
\begin{align*}
p_{11} + p_{12} + p_{13} + p_{14} + p_{15}+ p_{16} &= 1, \\
(-3)p_{11} + (-2)p_{12} +(-1) p_{13} + p_{14} + (2)p_{15}+ (3)p_{16} &= 0, \\
(-3)^2p_{11} +(-2)^2 p_{12} +(-1)^2 p_{13} + p_{14} +(2)^2 p_{15}+(3)^2 p_{16} &= 3, \\
(-3)^3p_{11} +(-2)^3 p_{12} + (-1)^3p_{13} + p_{14} +(2)^3 p_{15}+(3)^3 p_{16} &= 0, \\
(-3)^4p_{11} +(-2)^4 p_{12} + (-1)^4p_{13} + p_{14} +(2)^4 p_{15}+(3)^4 p_{16} &= 15, \\
(-3)^5p_{11} +(-2)^5 p_{12} + (-1)^5p_{13} + p_{14} +(2)^5 p_{15}+(3)^5 p_{16} &= 0.
\end{align*}
On solving these, $p_{11} = 0.05, p_{12} = 0.2, p_{13} = 0.25, p_{14}= 0.25, p_{15}=0.2, p_{16}=0.05.$

Therefore, by Theorem 1, the vertex energy of \( v_1 \) in \( G \) is given by
\begin{align*}
\mathcal{E}_{G}(v_1) &= p_{11}|\lambda_1| + p_{12}|\lambda_2| + p_{13}|\lambda_3|+ p_{14}|\lambda_4|+p_{15}|\lambda_5|+p_{16}|\lambda_6|\\
&= 0.05|-3|+0.2|-2|+0.25|-1|+0.25|1|+0.2|2|+0.05|3| \\
& = 1.6.
\end{align*}

This implies that \( \mathcal{\mathcal{E}}_{(G)}(v_i) = 1.6 \) holds for all \( i = 1, 2, \ldots, 20 \).
\end{proof}

\subsection{Vertex energies of Tutte-Coxeter graph}
\begin{theorem}
Let $G$ be the Tutte-Coxeter graph with vertex set \( V(G) = \{v_1, v_2, \ldots, v_{30}\} \). Then $\mathcal{E}_{G}(v_{i}) \approx 1.4, \, 1 \leq i \leq 30$.    
\end{theorem}
\begin{proof}
The characteristic polynomial of the Tutte-Coxeter graph \( G \) \cite{10} is given by 
\[
\phi(G; \lambda) = (\lambda +3)(\lambda +2)^9(\lambda-0)^{10}(\lambda -2)^9(\lambda - 3).
\]    

Let \( p_{11}, p_{12}, p_{13}, p_{14}, p_{15}\) denote the weights of the vertex \( v_1 \) in the graph \( G \). Then, according to Theorem 2, the following system of linear equations is obtained:
\begin{align*}
p_{11} + p_{12} + p_{13} + p_{14} + p_{15} &= 1, \\
(-3) \, p_{11} + (-2) \, p_{12} + 0 \, p_{13} + (2) \, p_{14} + (3) \, p_{15} &= 0, \\
(-3)^2 \, p_{11} + (-2)^2 \, p_{12} + 0 \, p_{13} + (2)^2 \, p_{14} + (3)^2 \, p_{15}&= 3, \\
(-3)^3\, p_{11} + (-2)^3 \, p_{12} + 0 \, p_{13} + (2)^3 \, p_{14} + (3)^3 \,p_{15}&= 0, \\
(-3)^4\, p_{11} + (-2)^4 \, p_{12} + 0 \, p_{13} + (2)^4 \, p_{14} +(3)^4 \, p_{15}&= 15.
\end{align*}
On solving these, $p_{11} = 0.03333, p_{12} = 0.3, p_{13} = 0.33333, p_{14} = 0.3, p_{15}=0.03333.$

Therefore, by Theorem 1, the vertex energy of \( v_1 \) in \( G \) is given by
\begin{align*}
\mathcal{E}_{G}(v_1) &= p_{11}|\lambda_1| + p_{12}|\lambda_2| + p_{13}|\lambda_3| +p_{14}|\lambda_4|+p_{15}|\lambda_5|\\
&= 0.03333 |-3| + 0.3|-2| + 0.33333|0|+ 0.3|2|+0.03333|3| \\
& \approx 1.4.
\end{align*}

This implies that \( \mathcal{\mathcal{E}}_{(G)}(v_i) \approx 1.4 \) holds for all \( i = 1, 2, \ldots, 30 \).
\end{proof}

\subsection{Vertex energies of Heawood graph}
\begin{theorem}
Let $G$ be the Heawood graph with vertex set \( V(G) = \{v_1, v_2, \ldots, v_{14}\} \). Then $\mathcal{E}_{G}(v_{i}) \approx 1.6407, \, 1 \leq i \leq 14$.    
\end{theorem}
\begin{proof}
The characteristic polynomial of the Heawood graph \( G \) \cite{11} is given by 
\[
\phi(G; \lambda) = (\lambda +3)(\lambda + \sqrt{2})^6(\lambda - \sqrt{2})^6(\lambda - 3).
\]    

Let \( p_{11}, p_{12}, p_{13}, p_{14}\) be the weights of the vertex \( v_1 \) in \( G \). Then, by Theorem 2, we have the following system of linear equations:
\begin{align*}
p_{11} + p_{12} + p_{13} + p_{14} &= 1, \\
(-3) \, p_{11} + (-\sqrt{2}) \, p_{12} + \sqrt{2} \, p_{13} + 3 \, p_{14} &= 0, \\
(-3)^2 \, p_{11} + (-\sqrt{2})^2 \, p_{12} + (\sqrt{2})^2 \, p_{13} + (3^2) \, p_{14} &= 3, \\
(-3)^3\, p_{11} + (-\sqrt{2})^3 \, p_{12} + (\sqrt{2})^3 \, p_{13} + (3^3) \, p_{14} &= 0.
\end{align*}
On solving these, $p_{11} = 0.07143, \quad p_{12} = 0.42857, \quad p_{13} = 0.42857, \quad p_{14} = 0.07143.$

Therefore, by Theorem 1, the vertex energy of \( v_1 \) in \( G \) is given by
\begin{align*}
\mathcal{E}_{G}(v_1) &= p_{11}|\lambda_1| + p_{12}|\lambda_2| + p_{13}|\lambda_3| +p_{14}|\lambda_4|\\
&= 0.07143 |-3| + 0.42857|-\sqrt{2}| + 0.42857|\sqrt{2}|+ 0.07143|3| \\
& \approx 1.6407.
\end{align*}

This implies that \( \mathcal{\mathcal{E}}_{(G)}(v_i) \approx 1.6407 \) holds for all \( i = 1, 2, \ldots, 14 \).
\end{proof}

\subsection{Vertex energies of Shrikhande graph}
\begin{theorem}
Let $G$ be the Shrikhande graph with vertex set \( V(G) = \{v_1, v_2, \ldots, v_{16}\} \). Then $\mathcal{E}_{G}(v_{i}) = 2.25, \, 1 \leq i \leq 16$.    
\end{theorem}
\begin{proof}
The characteristic polynomial of the Shrikhande graph \( G \) \cite{12} is given by 
\[
\phi(G; \lambda) = (\lambda +2)^9(\lambda -2)^6(\lambda-6).
\]    

Let \( p_{11}, p_{12}, p_{13}\) be the weights of the vertex \( v_1 \) in \( G \). Then, by Theorem 2, we have the following system of linear equations:
\begin{align*}
p_{11} + p_{12} + p_{13} &= 1, \\
(-2) \, p_{11} + (2) \, p_{12} + 6 \, p_{13} &= 0, \\
(-2)^2 \, p_{11} + (2)^2 \, p_{12} + (6)^2 \, p_{13} + &= 6.
\end{align*}
On solving these, $p_{11} = 0.5625, p_{12} = 0.375, p_{13} = 0.0625.$

Therefore, by Theorem 1, the vertex energy of \( v_1 \) in \( G \) is given by
\begin{align*}
\mathcal{E}_{G}(v_1) &= p_{11}|\lambda_1| + p_{12}|\lambda_2| + p_{13}|\lambda_3| \\
&= 0.5625 |-2| + 0.375|2| + 0.0625|6| \\
& = 2.25.
\end{align*}

This implies that \( \mathcal{\mathcal{E}}_{(G)}(v_i) = 2.25 \) holds for all \( i = 1, 2, \ldots, 16 \).
\end{proof}

\subsection{Vertex energies of Petersen graph}
\begin{theorem}
Let $G$ be the Petersen graph with vertex set \( V(G) = \{v_1, v_2, \ldots, v_{10}\} \). Then $\mathcal{E}_{G}(v_{i})=1.6, \, 1 \leq i \leq 10$.     
\end{theorem}
\begin{proof}
The characteristic polynomial of the Petersen graph \( G \) \cite{13} is given by 
\[
\phi(G; \lambda) = (\lambda - 3)(\lambda + 2)^4(\lambda - 1)^5.
\]

Let \( p_{11}, p_{12}, p_{13}\) be the weights of the vertex \( v_1 \) in \( G \). Then, by Theorem 2, we have the following system of linear equations:
\begin{align*}
p_{11} + p_{12} + p_{13} &= 1, \\
3 \, p_{11} + (-2) \, p_{12} + p_{13} &= 0, \\
3^2\, p_{11} + (-2)^2 p_{12} + p_{13} &= 3.
\end{align*}
Solving this system, we obtain: $p_{11} = 0.1, \quad p_{12} = 0.4, \quad p_{13} = 0.5.$

Therefore, by Theorem 1, the vertex energy of \( v_1 \) in \( G \) is given by
\begin{align*}
\mathcal{E}_{G}(v_1) &= p_{11}|\lambda_1| + p_{12}|\lambda_2| + p_{13}|\lambda_3| \\
&= 0.1 |3| + (0.4)|-2| + (0.5)|1| \\
& = 1.6.
\end{align*}

This implies that \( \mathcal{\mathcal{E}}_{(G)}(v_i) = 1.6 \) holds for all \( i = 1, 2, \ldots, 10 \).  
\end{proof}

\section{Conclusion}
In this paper, we studied the vertex energy of regular graphs, such as the Frucht graph, Desargues graph, Tutte-Coxeter graph, Heawood graph, Shrikhande graph, and Petersen graph. We analyzed their adjacency matrices to calculate their vertex energies. Although we made progress, vertex energy is still not widely studied in spectral graph theory.
\newpage
A key limitation in our study lies in the challenge of determining the spectral contribution or weight associated with each vertex. While the vertex energy is defined via the diagonal entries of $|A| = (AA^*)^{1/2}$, interpreting or isolating the influence of individual eigenvalues or eigenvectors on a specific vertex remains nontrivial. These weights are often difficult to compute analytically, especially for large or highly symmetric graphs, and require careful numerical approximation or matrix decomposition techniques. Furthermore, the lack of closed-form expressions for vertex weights limits our ability to generalize results or compare across different graph families.

Future research could explore vertex energy in other types of graphs, such as strongly regular and distance-regular graphs. It would also be useful to study how vertex energy relates to other graph properties, like eigenvalue gaps, resistance distance, and nodal domains. Since this topic is not well explored, further studies could lead to discoveries and applications in fields like chemistry, physics, and network science. A better understanding of vertex weights and their relation to graph symmetries could pave the way for new invariants and structural descriptors in spectral graph theory.


\end{document}